\documentclass[10pt,a4paper]{article}
\usepackage{amsmath}
\usepackage{amsthm}
\usepackage{amsfonts}
\usepackage{amssymb}
\usepackage{stmaryrd}
\usepackage{latexsym}

\usepackage{multirow}
\usepackage{array}
\newcolumntype{P}[1]{>{\centering\arraybackslash}p{#1}}
\usepackage{longtable}

\newtheorem{theorem}{Theorem}
\newtheorem{lemma}{Lemma}
\newtheorem{corollary}{Corollary}

\theoremstyle{definition}
\newtheorem{definition}{Definition}
\newtheorem{remark}{Remark}

\begin{document}

 \tolerance2500

\title{\Large{\textbf{Units in quasigroups  with classical Bol-Moufang type  identities}}}
\author{\normalsize {Natalia Didurik, Victor Shcherbacov}
}

 \maketitle

\begin{abstract}
We prolong Kunen's research about  existence of units (left, right, two-sided) in quasigroups with   classical Bol-Moufang type identity.
  These Bol-Moufang type  identities  are listed in  Fenvesh's article \cite{Fenyves_1}.

\medskip


\medskip

\end{abstract}

\bigskip

\section{Introduction}

We research the existence of units in quasigroups with classical  Bol-Moufang type identity.

Binary groupoid $(G, \ast)$ is  a non-empty set $G$ together with a binary operation \lq\lq $\ast$\rq\rq \, which is defined on the set $G$.

Groupoid $(Q, \ast)$ is called a quasigroup, if the following conditions are true \cite{VD}:
$(\forall u, v \in  Q) (\exists ! \, x, y \in  Q) (u * x = v  \, \& \,  y * u = v)$. This is a so-called existential  definition \cite{MUFANG}.

The following definition is called an equational definition of quasigroup.

\begin{definition}  \label{EQUT_QUAS_DEF} \cite{BIRKHOFF_ENG, BIRKHOFF, EVANS_49}.
 A groupoid $(Q, \cdot)$ is called a quasigroup if on the set $Q$ there exist
operations \lq\lq $\backslash$" and \lq\lq $/$" such that in the algebra $(Q, \cdot, \backslash, /)$
identities
\begin{equation}
x\cdot(x \backslash y) = y, \label{(1)}
\end{equation}
\begin{equation}
(y / x)\cdot x = y, \label{(2)}
\end{equation}
\begin{equation}
x\backslash (x \cdot y) = y, \label{(3)}
\end{equation}
\begin{equation}
(y \cdot x)/ x = y, \label{(4)}
\end{equation}
\end{definition}
 are fulfilled.

In any quasigroup $(Q, \cdot, \backslash, /)$ the following identities are true:
\begin{equation}
x/(y\backslash x) = y,  \label{T}
\end{equation}
\begin{equation}
(x/y)\backslash x = y. \label{R}
\end{equation}

Standard introduction in quasigroup theory is given in \cite{VD, HOP, 2017_Scerb}.

Let $(Q, \cdot)$ be a quasigroup.

\begin{definition}
An element $i \in Q$ is an \emph{idempotent} of $(Q, \cdot)$ if and only if $i \cdot i = i$\/.

An element $f \in Q$  is a \emph{left unit}  of $(Q,  \cdot)$ if and only if  $f \cdot x =x$ for all $x\in Q$\/.

An element $e \in Q$  is a \emph{right  unit}  of $(Q, \cdot)$ if and only if $x \cdot e =x$ for all $x \in Q$\/.

An element $e \in Q$  is a \emph{(two--sided)  unit}  of $(Q, \cdot)$ if and only if it is both left and right unit.

An element $m \in Q$ is a \emph{middle unit}  of $(Q; \cdot)$ if and only if $x \cdot x  = m$ for all $x \in Q$\/.
\end{definition}

\begin{definition}
A quasigroup is a left (right) loop if it has a left (right) unit.

A quasigroup is a loop if it has both left and right units.
\end{definition}

An identity based on a single binary operation is of Bol-Moufang type if \lq\lq both sides consist of the same three different letters taken in the same order but one of them occurs twice on each side\rq\rq \cite{Fenyves_1}. We use list and denotation of 60 Bol-Moufang type identities given in \cite{JAIEOLA_2}.  

\begin{remark}
There exist other (\lq\lq more general\rq\rq) definitions of Bol-Moufang type identities and, therefore, other lists and classifications of such identities \cite{AKHTAR, Cote}.
\end{remark}

 Quasigroups and loops, in which a Bol-Moufang type identity is  true, are  central and classical objects of Quasigroup Theory. We recall, works of R.~Moufang, G.~Bol, R.~H.~Bruck, V.~D.~Belousov, K. Kunen,  S.~Gagola   and of many other mathematicians are devoted to the study of quasigroups and loops with Bol-Moufang type identities \cite{MUFANG, BRUCK_60_1, VD, Fenyves_1, KUNEN_96,  Gagola_10, Gagola_11, Ph_2005, Cote, AKHTAR}.

Reformulating  title of Gagola's article  \lq\lq How and why Moufang loops behave like groups\rq\rq  \cite{Gagola_11} we can say that quasigroups with Bol-Moufang identities \lq\lq behave  like groups\rq\rq. This is one of  reasons why we study these quasigroup classes. Notice, information on right and left unit elements in quasigroups with any Bol and any  Moufang identity is known \cite{VD, GALKIN,  KUNEN, IZ_SCERB, VS_VI}.

In some cases we have used Prover9 \cite{MAC_CUNE_PROV} and Mace4 \cite{MAC_CUNE_MACE} for  finding the  proofs  and constructing counterexamples. The nearest article to our paper is Kunen's  article \cite{ KUNEN_96}. Notice, Professor Kunen in his researches have used some versions of  Prover  and Mace.

\section{Results}

\subsection{$(12)$-parastrophes of identities}

We recall, $(12)$-parastroph of groupoid $(G, \cdot)$ is a groupoid $(G, \ast)$ in which  operation \lq\lq $\ast$ \rq\rq \, is obtained by the following rule:
\begin{equation} \label{HSTY_1}
x\ast y = y\cdot x.
\end{equation}

It is clear that   for any groupoid $(G, \cdot)$   there exists its $(12)$-parastroph groupoid $(G, \ast)$.

Cayley table of groupoid $(G, \ast)$ is a mirror image of the Cayley table of groupoid $(G, \cdot)$ relative to main diagonal.
Notice, for any binary quasigroup there exist  five its parastrophes \cite{VD, HOP, 2017_Scerb} more.

Suppose that an identity $F$ is true in groupoid $(G, \cdot)$.
Then we can obtain $(12)$-parastrophic identity $F^{\ast}$ of the identity $F$ replacing the operation \lq\lq $\cdot$\rq\rq \, with  the operation  \lq\lq $\ast$\rq\rq \, and changing the order of variables using rule (\ref{HSTY_1}).
\begin{remark}
In quasigroup case,  similarly to $(12)$-parastrophe identity other parastrophe identities can be  defined. See \cite{Scer_Shcherb_2} for details.
\end{remark}

It is clear that an identity $F$ is true in groupoid $(G, \cdot)$ if and only if in groupoid $(Q, \ast)$ identity $F^{\ast}$ is true.
The following lemma is evident. See  \cite{HSTY_IMI}.

\begin{lemma}  \label{12-parastrophes_1}
For classical Bol-Moufang type identities over groupoids the following equalities are true:

$(F_1)^{\ast} = F_3$, $(F_2)^{\ast} = F_4$, $(F_5)^{\ast} = F_{10}$, $(F_6)^{\ast} = F_6$, $(F_7)^{\ast} = F_8$,
$(F_9)^{\ast} = F_9$,   $(F_{11})^{\ast} = F_{24}$, $(F_{12})^{\ast} = F_{23}$,    $(F_{13})^{\ast} = F_{22}$,
$(F_{14})^{\ast} = F_{21}$, $(F_{15})^{\ast} = F_{30}$, $(F_{16})^{\ast} = F_{29}$, $(F_{17})^{\ast} = F_{27}$, $(F_{18})^{\ast} = F_{28}$,
$(F_{19})^{\ast} = F_{26}$, $(F_{20})^{\ast} = F_{25}$,  $(F_{31})^{\ast} = F_{34}$, $(F_{32})^{\ast} = F_{33}$, $(F_{35})^{\ast} = F_{40}$,
 $(F_{36})^{\ast} = F_{39}$, $(F_{37})^{\ast} = F_{37}$, $(F_{38})^{\ast} = F_{38}$, $(F_{41})^{\ast} = F_{53}$, $(F_{42})^{\ast} = F_{54}$,
 $(F_{43})^{\ast} = F_{51}$,$(F_{44})^{\ast} = F_{52}$, $(F_{45})^{\ast} = F_{60}$, $(F_{46})^{\ast} = F_{56}$, $(F_{47})^{\ast} = F_{58}$,
 $(F_{48})^{\ast} = F_{57}$, $(F_{49})^{\ast} = F_{59}$, $(F_{50})^{\ast} = F_{55}$.
\end{lemma}

For quasigroups, analogue of Theorem \ref{12-parastrophes_1} is given in \cite{KUNEN_96}.

It is easy to see that any group satisfies any of identities $F_1-F_{60}$. Therefore the cyclic group $Z_3$ is a counter-example to proposition that \lq\lq there exist a quasigroup with an  identity from list of identities  $F_1-F_{60}$ and which has middle unit\rq\rq.

\begin{lemma} \label{12-parastrophes_2}
If a quasigroup $(Q, \cdot)$ has a left (right) unit element, then $(12)$-parastrophe of quasigroup $(Q, \cdot)$ has right (left) unit element.
\end{lemma}
\begin{proof}
It is easy to see.
\end{proof}

\begin{lemma} \label{Bijection}
If in quasigroup $(Q, \cdot)$ the variables $x$ and $y$ run through the whole set $Q$ ($x\neq y$), then the term $x\cdot y$ runs through the whole set $Q$, too.
\end{lemma}
\begin{proof}
If we fix variable $x$, for example, $x=a$, $a\in Q$, then   even the term $ay$ runs through the whole set $Q$.
 See, also, \cite{2017_Scerb}.
\end{proof}
Notice, situation is another when $x=y$. See, for example,  identity $F_{42}$, Theorem \ref{Identity_F42}.

\begin{lemma} \label{Group_Unit}
If in quasigroup $(Q, \cdot)$ identity of associativity  is  true, then this quasigroup is a group.
\end{lemma}
\begin{proof}
In order to show for readers of this paper  methods of the proofs, which  usually are  used in this paper,   we give the well known standard proof. See, for example, \cite{Kurosh}.

Here we use Lemma \ref{Bijection}.
In identity of  associativity $x\cdot yz = xy\cdot z$ we put $x = f_y$ ($f_y y = y$). Then we have
$f_y\cdot yz = yz$. Therefore element $f_y$ is a left unit  of quasigroup $(Q, \cdot)$.

In identity of associativity $x\cdot yz = xy\cdot z$ we put $z = e_y$ ($y e_y = y$). Then we have
$xy = xy\cdot e_z$. Therefore element $e_z$ is a right unit  of quasigroup $(Q, \cdot)$.
\end{proof}

\begin{lemma} \label{Group_parastroph}
If  quasigroup $(Q, \cdot)$ is a group, then quasigroup $(Q, \overset{(12)}{\cdot})$ is a group.
\end{lemma}
\begin{proof}
This  is mathematical folklore.
\end{proof}

\begin{theorem} \label{Identity_F_1}
Quasigroup $(Q, \cdot)$ with identity $F_1$ is a group.
\end{theorem}
\begin{proof}
Using Lemma \ref{Bijection} we can rewrite identity $F_1$ ($xy\cdot zx = (xy\cdot z)x $) in the form $t\cdot zx = tz\cdot x,$ where $t=xy$. We obtain identity of associativity. By Lemma \ref{Group_Unit}, quasigroup $(Q, \cdot)$ is a loop, moreover, it is a group. Then quasigroup $(Q, \cdot)$ is a group.
\end{proof}

\begin{corollary} \label{Identity_F_3}
Quasigroup $(Q, \cdot)$ with identity $F_3$ is a group.
\end{corollary}
\begin{proof}
By Theorem \ref{12-parastrophes_1} $F_3 = (F_1)^{\ast}$.
\end{proof}

\begin{theorem}
Quasigroup $(Q, \cdot)$ with identity $F_4$ and $F_2$ is a loop.
\end{theorem}
\begin{proof}
The fact that quasigroup $(Q, \cdot)$ with identity $F_4$ is a loop has  been  proved in \cite{GALKIN, VS_VI}.
The rest follows from Lemma \ref{12-parastrophes_2}, since by Theorem \ref{12-parastrophes_1} $(F_2)^{\ast} = F_4$.
\end{proof}
It is well known that there exist non-associative Moufang loops \cite{Chein_Hala}.

\begin{theorem} \label{F_5_F10}
Quasigroup $(Q, \cdot)$ with identities  $F_5$ and $F_{10}$ is a group.
\end{theorem}
\begin{proof}
It is well known that   any quasigroup $(Q, \cdot)$ has left and right cancellative  property \cite{VD, HOP, 2017_Scerb}.
Then from identity $(xy \cdot z)x = (x \cdot y z)x $ we have $xy \cdot z = x \cdot y z $. The last follows from
Lemma  \ref{Group_Unit}.
\end{proof}

\begin{theorem}
Quasigroup $(Q, \cdot)$ with identity $F_6$ $(xy \cdot z)x = x (y  \cdot z x)$ (extra identity) is a loop.
\end{theorem}
\begin{proof}
 See \cite{KUNEN_96}.
\end{proof}

There exist non-associative extra loops \cite{Kin_KUN_04}.

\begin{theorem} \label{Identity_F7}
Quasigroup $(Q, \cdot)$ with identity $F_7$ $(xy \cdot z)x = x (y z \cdot x)$  has  left and  has not right unit.
\end{theorem}
\begin{proof}
If we put $x := f_y$ in identity $F_7$, then we have $y  z \cdot f_y = f_y (y z \cdot f_y)$. It is clear that term  $(y z \cdot f_y)$ runs all elements of the set $Q$.

The following counterexample shows that quasigroup $(Q, \cdot)$ with identity $F_7$   has  not  right  unit.
\[
\begin{array}{c|ccc}
\cdot & 0 & 1 & 2\\
\hline
    0 & 1 & 2 & 0 \\
    1 & 0 & 1 & 2 \\
    2 & 2 & 0 & 1
\end{array}
\]
\end{proof}

\begin{corollary} \label{Identity F_8}
Quasigroup $(Q, \cdot)$ with identity $ F_8$ ($(x \cdot y z)x = x (y  \cdot z x) $)  has  right  and has not left  unit.
\end{corollary}
\begin{proof}
The proof follows from Theorem \ref{Identity_F7}, Lemma \ref{12-parastrophes_2} and the fact that $F_8 = (F_7)^{\ast}$ (Theorem  \ref{12-parastrophes_1}).
\end{proof}

\begin{theorem} \label{Identity_F9}
Quasigroup $(Q, \cdot)$ with identity $F_9$ $ ( (x \cdot y z)x = x (y   z \cdot x)$ )  has not left and  has not right unit.
\end{theorem}
\begin{proof}
The following quasigroup $(Q, \cdot)$ satisfies identity $F_9$ and it has not left and right unit.

\[
\begin{array}{c|ccc}
\cdot & 0 & 1 & 2\\
\hline
    0 & 1 & 0 & 2 \\
    1 & 0 & 2 & 1 \\
    2 & 2 & 1 & 0
\end{array}
\]
\end{proof}

\begin{theorem} \label{Identity_F11}
Quasigroup $(Q, \cdot)$ with identity $F_{11}$ $ xy\cdot xz = (xy \cdot x)z $  is a group.
\end{theorem}
\begin{proof}
We can rewrite identity $F_{11}$ $  xy\cdot xz = (xy \cdot x)z $ in the form $t\cdot xz = tx\cdot z,$ where $t=xy$. The rest follows from Lemma \ref{Group_Unit}.

\end{proof}

\begin{corollary} \label{Identity_F24}
Quasigroup $(Q, \cdot)$ with identity $F_{24}$ \;  $yx \cdot zx = y ( x \cdot z  x) $
 is a group.
\end{corollary}
\begin{proof}
By Theorem \ref{12-parastrophes_1} $F_{24} = (F_{11})^{\ast}$.
\end{proof}

\begin{theorem} \label{Identity_F12}
Quasigroup $(Q, \cdot)$ with identity $F_{12}$ ($ xy\cdot xz = (x\cdot y x)z$)  is a group.
\end{theorem}
\begin{proof}
If we put $x:=f_z$ in identity $F_{12}$, then $ f_zy\cdot f_zz = (f_z\cdot y f_z)z$. After cancellation in the last equality we have $y = y f_z$. The last equality means that quasigroup with identity $F_{12}$ has right unit.

If we put  $x=y=e$ in identity $F_{12}$, then we have $ ee\cdot ez = (e\cdot e e)z$, $e\cdot ez = ez$. Therefore quasigroup $(Q, \cdot)$ with identity $F_{12}$ has left unit, and, finally, this quasigroup is a loop.

If we put $z=1$ in identity $F_{12}$, then we have
\begin{equation}\label{F_12_1}  xy\cdot x = x  \cdot yx
\end{equation}

If we apply equality (\ref{F_12_1}) to identity $F_{12}$, then we have
\begin{equation}\label{F_12_2}
xy\cdot xz = (xy\cdot  x)z.
\end{equation}
If we denote term $xy$ by the letter $t$, then we can re-write identity (\ref{F_12_2}) in the form
\begin{equation}\label{F_12_3}
t\cdot xz = t  x \cdot z.
\end{equation}
\end{proof}

\begin{corollary} \label{identity_F_23}
Quasigroup $(Q, \cdot)$ with identity $F_{23}$  $ yx \cdot zx = y ( x z \cdot x)  $
 is a group.
\end{corollary}
\begin{proof}
By Theorem \ref{12-parastrophes_1} $F_{23} = (F_{12})^{\ast}$.
\end{proof}

\begin{theorem} \label{Identity_F13}
Quasigroup $(Q, \cdot)$ with identity $F_{13}$ $ xy\cdot xz = x(y x\cdot z)$  is a loop.
\end{theorem}
\begin{proof}
See \cite{Kin_KUN_04}.
\end{proof}

\begin{corollary} \label{identity_F_22}
Quasigroup $(Q, \cdot)$ with identity $F_{22}$  $ yx \cdot zx = (y \cdot x z)x $
 is a loop.
\end{corollary}
\begin{proof}
By Theorem \ref{12-parastrophes_1} $F_{22} = (F_{13})^{\ast}$.
\end{proof}

Examples of non-associative loops with identities $F_{13}$ and $F_{22}$ are given in \cite{Kin_KUN_04},

\begin{theorem} \label{Identity_F14}
Quasigroup $(Q, \cdot)$ with identity $F_{14}$ $ xy\cdot xz = x(y \cdot x z) $  is a group.
\end{theorem}
\begin{proof}
We can put $t= xz$. Further proof is similar to the proof  of Theorem  \ref{Identity_F11}.
\end{proof}

\begin{corollary} \label{identity_F_21}
Quasigroup $(Q, \cdot) $ with identity $F_{21}$  $ yx \cdot zx = (yx\cdot z)x $
 is a group.
\end{corollary}
\begin{proof}
By Theorem \ref{12-parastrophes_1} $F_{21} = (F_{14})^{\ast}$.
\end{proof}

\begin{theorem} \label{Identity_F15}
Quasigroup $(Q, \cdot)$ with identity $F_{15}$ $ (xy \cdot x)z = (x\cdot y  x)z $  does not have left and right unit.
\end{theorem}
\begin{proof}
The following quasigroup $(Q, \cdot)$ satisfies identity $F_{15}$ and it has not left and right unit.

\[
\begin{array}{c|ccc}
\cdot & 0 & 1 & 2\\
\hline
    0 & 1 & 0 & 2 \\
    1 & 0 & 2 & 1 \\
    2 & 2 & 1 & 0
\end{array}
\]

\end{proof}

\begin{corollary} \label{identity_F_30}
Quasigroup $(Q, \cdot) $ with identity $F_{30}$  $ y(xz \cdot x) = y(x \cdot zx) $
does not have left and right unit element.
\end{corollary}
\begin{proof}
By Theorem \ref{12-parastrophes_1} $F_{30} = (F_{15})^{\ast}$.
\end{proof}

\begin{theorem} \label{Identity_F16}
Quasigroup $(Q, \cdot)$ with identity $F_{16}$ ($ (xy \cdot x)z = x(y  x\cdot z) $)  has left unit  and it has not right unit.
\end{theorem}
\begin{proof}
If we put  $x:=y$, $y:= y\backslash x$ in identity $F_{16}$, then we have $(y(y\backslash x) \cdot y) z = y((y\backslash x)y\cdot z)$ and after application of identity (\ref{(1)}) we have

\begin{equation} \label{F16_1}
xy\cdot z = y ((y\backslash x) y \cdot z).
\end{equation}
Using the operation $\backslash $ in equality (\ref{F16_1}) we have

\begin{equation} \label{F16_2}
(y\backslash x) y \cdot z  = y \backslash (xy\cdot z).
\end{equation}
If in equality (\ref{F16_2}) we put $xy =t$, then $x= t\slash y$,

\begin{equation} \label{F16_3}
(y\backslash (t\slash y))y \cdot z  = y \backslash (t\cdot z).
\end{equation}
If in equality (\ref{F16_3}) we put $y=t$  and apply identity (\ref{(3)}),
then we have
\begin{equation} \label{F16_4}
(y\backslash (y\slash y)) y \cdot z  =  z.
\end{equation}
The last equality demonstrates that quasigroup $(Q, \cdot)$ has left identity element.

The following quasigroup satisfies identity $F_{16}$ and it does not have right identity element.

\[
\begin{array}{c|ccc}
\cdot & 0 & 1 & 2\\
\hline
    0 & 1 & 2 & 0 \\
    1 & 0 & 1 & 2 \\
    2 & 2 & 0 & 1
\end{array}
\]
\end{proof}

\begin{corollary} \label{identity_F_29}
Quasigroup $(Q, \cdot) $ with identity $F_{29}$  $ (y\cdot xz)x  = y(x \cdot zx) $
does not have left unit  and it has right unit element.
\end{corollary}
\begin{proof}
By Theorem \ref{12-parastrophes_1} $F_{29} = (F_{16})^{\ast}$.
\end{proof}

\begin{theorem} \label{Identity_F17}
Quasigroup $(Q, \cdot)$ with identity $F_{17}$ (left \,\, Moufang)  $ (xy \cdot x)z = x(y \cdot  x z) $  is a loop.
\end{theorem}
\begin{proof}
 See \cite{IZ_SCERB, KUNEN, VS_VI}.

 \end{proof}

\begin{corollary} \label{identity_F_27}
Quasigroup $(Q, \cdot) $ with identity $F_{27}$  (right \,\, Moufang) $ (yx \cdot z)x = y (x\cdot  z  x) $
is a loop.
\end{corollary}
\begin{proof}
By Theorem \ref{12-parastrophes_1} $F_{27} = (F_{17})^{\ast}$.
\end{proof}

\begin{theorem} \label{Identity_F18}
Quasigroup $(Q, \cdot)$ with identity $F_{18}$  ( $ (x\cdot y  x)z = x(y   x \cdot z) $)  is a group.
\end{theorem}
\begin{proof}
Denote the term $y  x$ by variable $t$. Then identity $F_{18}$   takes the form $ xt\cdot z = x \cdot  t z $. The rest follows from
Lemma \ref{Group_Unit}.
 \end{proof}

\begin{corollary} \label{identity_F_28}
Quasigroup $(Q, \cdot) $ with identity $F_{28}$  $ (y \cdot xz)x = y(xz\cdot x) $
is a group.
\end{corollary}
\begin{proof}
By Theorem \ref{12-parastrophes_1} $F_{28} = (F_{18})^{\ast}$.
\end{proof}

\begin{theorem} \label{Identity_F19}
Quasigroup $(Q, \cdot)$ with identity $F_{19}$ (left \,\ Bol)   $ (x\cdot y  x)z = x(y \cdot  x  z) $ does not have left unit and it has right unit element.
\end{theorem}
\begin{proof}
Quasigroup $(Q, \cdot)$ with identity $F_{19}$ does not have left unit.

\[
\begin{array}{c|ccc}
\cdot & 0 & 1 & 2\\
\hline
   0 & 1 & 0 & 2 \\
    1 & 2 & 1 & 0 \\
    2 & 0 & 2 & 1
\end{array}
\]

Quasigroup $(Q, \cdot)$ with identity $F_{19}$ has right unit.
If in identity   $ (x\cdot y  x)z = x(y \cdot  x  z) $ we put $z:=e_x$, then we have $ (x\cdot y  x)e_x  = x\cdot y  x   $ for all $x, y \in Q$ \cite{2017_Scerb}.
\end{proof}

\begin{corollary} \label{identity_F_26}
Quasigroup $(Q, \cdot) $ with identity $F_{26}$ (right \,\, Bol ) $ (yx \cdot z)x = y (x z \cdot x) $
has left unit and does not have right unit.
\end{corollary}
\begin{proof}
By Theorem \ref{12-parastrophes_1} $F_{26} = (F_{19})^{\ast}$.
\end{proof}

\begin{theorem} \label{Identity_F20}
Quasigroup $(Q, \cdot)$ with identity $F_{20}$  ( $ x(y  x \cdot z) = x(y \cdot  x  z) $) is a group.
\end{theorem}
\begin{proof}
If we perform  cancellation from the left in identity  $ x(y  x \cdot z) = x(y \cdot  x  z) $, then we have $ y  x \cdot z = y \cdot  x  z $. Further we can apply Lemma \ref{Group_Unit}.
\end{proof}

\begin{corollary} \label{identity_F_25}
Quasigroup $(Q, \cdot) $ with identity $F_{25}$ ( $ (yx \cdot z)x = (y \cdot x z)x $)
is a group.
\end{corollary}
\begin{proof}
By Theorem \ref{12-parastrophes_1} $F_{25} = (F_{20})^{\ast}$.
\end{proof}

\begin{theorem} \label{Identity_F31}
Quasigroup $(Q, \cdot)$ with identity $F_{31}$    ($ yx \cdot xz = (yx \cdot x)z $) is a group.
\end{theorem}
\begin{proof}
If we denote term $yx$ as $t$, then from identity $F_{31}$   we obtain identity of associativity $ t \cdot xz = t \cdot xz $.
\end{proof}

\begin{corollary} \label{identity_F_34}
Quasigroup $(Q, \cdot) $ with identity $F_{34}$ ($ yx \cdot xz = y( x \cdot  x z) $)
is a group.
\end{corollary}
\begin{proof}
By Theorem \ref{12-parastrophes_1} $F_{34} = (F_{31})^{\ast}$.
\end{proof}

\begin{theorem} \label{Identity_F32}
Quasigroup $(Q, \cdot)$ with identity $F_{32}$    ($ yx \cdot xz = (y\cdot x x)z $) is a group.
\end{theorem}
\begin{proof}
We prove that quasigroup $(Q, \cdot)$ with identity $F_{32}$ has left unit. In identity  $F_{32}$ $ yx \cdot xz = (y\cdot x x)z $ we change  $y \rightarrow x\slash (zz)$, $x\rightarrow z$, $z\rightarrow y$ and obtain
\begin{equation} \label{F32_1}
((x\slash zz)zz)y = ((x\slash zz)z)\cdot zy,
\end{equation}
and after application of identity (\ref{(2)}) we have
\begin{equation} \label{F32_2}
xy = ((x\slash zz)z)\cdot zy.
\end{equation}
If we denote term $zy$ by the letter $t$, then $y=z\backslash t$. And we can rewrite equality
(\ref{F32_2}) in the form
\begin{equation} \label{F32_3}
x(z\backslash t) = ((x\slash zz)z)\cdot t.
\end{equation}
If we put  $x=z$ in equality (\ref{F32_3}), then we have
\begin{equation} \label{F32_4}
x(x\backslash t) \overset{(\ref{(1)})}{=} t = ((x\slash xx)x)\cdot t.
\end{equation}
We have proved that quasigroup $(Q, \cdot)$ with identity $F_{32}$ has left unit $f = ((x\slash xx)x)$.


If we substitute  $x=e$ in identity $F_{32}$, then we have
\begin{equation} \label{F32_5}
y \cdot ez = yz, ez=z, e=f.
\end{equation}
Since in quasigroup $(Q, \cdot)$ there exists left unit, then there exists
right unit too.

If in identity $F_{32}$ we put $z=e$, then we have
\begin{equation} \label{F32_6}
yx \cdot x = y\cdot x x.
\end{equation}
If we apply identity (\ref{F32_6}) to the right side of identity $F_{32}$, then we have
\begin{equation} \label{F32_7}
yx \cdot xz = (yx\cdot x)z.
\end{equation}
In order to prove that quasigroup with identity $F_{32}$ is a group  we can denote term $yx$ in identity (\ref{F32_7})  by the letter $t$ and use Lemmas \ref{Bijection} and  \ref{Group_Unit}.
\end{proof}

\begin{corollary} \label{identity_F_33}
Quasigroup $(Q, \cdot) $ with identity $F_{33}$ $ yx \cdot xz = y( x x \cdot z) $ is a group.
\end{corollary}
\begin{proof}
By Theorem \ref{12-parastrophes_1} $F_{33} = (F_{32})^{\ast}$.
\end{proof}

\begin{theorem} \label{Identity_F35}
Quasigroup $(Q, \cdot)$ with identity $F_{35}$    $ (yx \cdot x) z = (y\cdot xx)z $ does not have left unit and it has right unit element.
\end{theorem}
\begin{proof}
Prove that quasigroup $(Q, \cdot)$ with identity $F_{35}$ has right unit.
In identity $F_{35}$ we make cancellation from the right side and we have:
\begin{equation} \label{F35_0}
yx \cdot x = y\cdot xx.
\end{equation}
If in identity (\ref{F35_0}) we put $yx=t$, then $t\slash x= y$. Using this substitution we can rewrite identity (\ref{F35_0}) as follows:
\begin{equation} \label{F35_3}
tx = (t\slash x) \cdot xx.
\end{equation}

Suppose that    $x:= y \backslash x$ in identity $F_{35}$. Then  we have:
\begin{equation} \label{F35_1}
(y(y \backslash x) \cdot (y \backslash x)) z \overset{(\ref{(1)})}{=} (x \cdot (y \backslash x)) z =  (y\cdot (y \backslash x)(y \backslash x))z.
\end{equation}
After cancellation in identity (\ref{F35_1}) from the right side we have:
\begin{equation} \label{F35_2}
x \cdot (y \backslash x)  =  y\cdot (y \backslash x)(y \backslash x).
\end{equation}
From identity (\ref{F35_2}) using the operation of right division \lq\lq $\backslash $ \rq\rq we have
\begin{equation} \label{F35_4}
 (y \backslash x)(y \backslash x)   =  y\backslash (x \cdot (y \backslash x)).
\end{equation}
If we put  $x=y$ in identity (\ref{F35_4}), then we have:
\begin{equation} \label{F35_5}
 (x \backslash x)(x \backslash x)   =  x\backslash (x (x \backslash x)) \overset{(\ref{(1)})}{=} x \backslash x.
\end{equation}
In identity (\ref{F35_3}) we put $x:= z\backslash z$ for some $z\in Q$.
Then we have:
\begin{equation} \label{F35_6}
\begin{split}
& t(z\backslash z) = (t\slash (z\backslash z)) \cdot (z\backslash z)(z\backslash z) \overset{(\ref{F35_5})} {=}\\
&(t\slash (z\backslash z)) \cdot (z\backslash z)\overset{(\ref{(2)})}{=} t.
\end{split}
\end{equation}

The following example demonstrates that quasigroup $(Q, \cdot)$ with identity $F_{35}$ does not have left unit.

\[
\begin{array}{c|cccccc}
\cdot & 0 & 1 & 2 & 3 & 4 & 5\\
\hline
    0 & 1 & 0 & 2 & 3 & 4 & 5 \\
    1 & 0 & 1 & 3 & 2 & 5 & 4 \\
    2 & 3 & 2 & 4 & 5 & 0 & 1 \\
    3 & 2 & 3 & 5 & 4 & 1 & 0 \\
    4 & 5 & 4 & 0 & 1 & 2 & 3 \\
    5 & 4 & 5 & 1 & 0 & 3 & 2
\end{array}
\]
\end{proof}

\begin{corollary} \label{identity_F_40}
Quasigroup $(Q, \cdot) $ with identity $F_{40}$ $ y (xx\cdot z)  = y (x\cdot xz) $ has left unit and it does not have right unit element.
\end{corollary}
\begin{proof}
By Theorem \ref{12-parastrophes_1} $F_{35} = (F_{40})^{\ast}$.
\end{proof}

\begin{theorem} \label{Identity_F36}
Quasigroup $(Q, \cdot)$ with identity $F_{36}$ (RC \,\, identity)    $ (yx \cdot x) z = y(xx \cdot z) $ has left unit and it does not have right unit element.
\end{theorem}
\begin{proof}

Quasigroup $(Q, \cdot)$ with identity $F_{36}$ has left  unit. If we put $y:=f_x$ in identity $ (yx \cdot x) z = y(xx \cdot z) $, then we have
\begin{equation} \label{F36_1}
(f_xx \cdot x) z = f_x(xx \cdot z), xx \cdot  z = f_x(xx \cdot z).
\end{equation}
From  equalities (\ref{F36_1}) it follows that quasigroup $(Q, \cdot)$ with identity $F_{36}$ has left unit.

The following example demonstrates that right identity element does not exist.
\[
\begin{array}{c|ccc}
\cdot & 0 & 1 & 2\\
\hline
 0 & 1 & 2 & 0 \\
    1 & 0 & 1 & 2 \\
    2 & 2 & 0 & 1
\end{array}
\]
\end{proof}

\begin{corollary} \label{identity_F_39}
Quasigroup $(Q, \cdot) $ with identity $F_{39}$ (LC\,\, identity ) $ (y \cdot xx) z = y(x\cdot xz) $ does not have left unit and it  has right unit element.
\end{corollary}
\begin{proof}
By Theorem \ref{12-parastrophes_1} $F_{39} = (F_{36})^{\ast}$.
\end{proof}

\begin{theorem} \label{Identity_F37}
Quasigroup $(Q, \cdot)$ with identity $F_{37}$ (C  identity)    $ (yx \cdot x)z = y(x\cdot xz) $ does not have left unit and it does not have right unit.
\end{theorem}
\begin{proof}
We give necessary example.
\[
\begin{array}{c|ccc}
\cdot & 0 & 1 & 2\\
\hline
 0 & 1 & 0 & 2 \\
    1 & 0 & 2 & 1 \\
    2 & 2 & 1 & 0
\end{array}
\]
\end{proof}

\begin{theorem} \label{Identity_F38}
Quasigroup $(Q, \cdot)$ with identity $F_{38}$     $ (y \cdot xx) z = y(xx \cdot z) $ is a loop.
\end{theorem}
\begin{proof}
In identity $F_{38}$ we put $y:= f_{xx}$. Then we have
\begin{equation} \label{F38_1}
(f_{xx} \cdot xx) z = f_{xx}(xx \cdot z), xx \cdot z = f_{xx}(xx \cdot z).
\end{equation}

In identity $F_{38}$ we put $z:= e_{xx}$. Then we have
\begin{equation} \label{F38_1}
(y \cdot xx) e_{xx} = y(xx \cdot e_{xx}), (y \cdot xx) e_{xx} = y \cdot xx.
\end{equation}
Therefore in quasigroup $(Q, \cdot)$ with identity $F_{38}$  there exists right unit.
\end{proof}

The following example demonstrates that quasigroup with identity $F_{38}$ is not  associative; loop with identity $F_{38}$ has  not Lagrange property.

\[
\begin{array}{c|ccccc}
\cdot & 0 & 1 & 2 & 3 & 4\\
\hline
    0 & 0 & 1 & 2 & 3 & 4 \\
    1 & 1 & 0 & 3 & 4 & 2 \\
    2 & 2 & 4 & 0 & 1 & 3 \\
    3 & 3 & 2 & 4 & 0 & 1 \\
    4 & 4 & 3 & 1 & 2 & 0
\end{array}
\]

\begin{theorem} \label{Identity_F41}
Quasigroup $(Q, \cdot)$ with identity $F_{41}$  (LC  identity   $ xx \cdot yz = (x \cdot xy)z $) is a loop.
\end{theorem}
\begin{proof}
We prove that quasigroup $(Q, \cdot)$ with identity $F_{41}$  has left   identity element. In identity $F_{41}$ we put $y:= e_{x}$. Then we have:
\begin{equation} \label{F41_1}
xx \cdot e_xz = (x \cdot xe_x)z, xx \cdot e_xz = x x \cdot z, e_xz =z.
\end{equation}

We prove that quasigroup $(Q, \cdot)$ with identity $F_{41}$  has right   identity element. In identity $F_{41}$ we put $y:= f_{z}$. Then we have:
\begin{equation} \label{F41_1}
xx \cdot f_z z = (x \cdot xf_z)z, \, xx \cdot z = (x  \cdot xf_z)  z, \, xx  = x  \cdot xf_z,\; x  = xf_z.
\end{equation}
\end{proof}
The following example demonstrates that loop with identity $F_{41}$ is not  associative.
\[
\begin{array}{c|cccccc}
\cdot & 0 & 1 & 2 & 3 & 4 & 5\\
\hline
    0 & 0 & 1 & 2 & 3 & 4 & 5 \\
    1 & 1 & 0 & 3 & 5 & 2 & 4 \\
    2 & 2 & 5 & 0 & 4 & 1 & 3 \\
    3 & 3 & 4 & 1 & 0 & 5 & 2 \\
    4 & 4 & 3 & 5 & 2 & 0 & 1 \\
    5 & 5 & 2 & 4 & 1 & 3 & 0
\end{array}
\]

\begin{corollary} \label{identity_F_53}
Quasigroup $(Q, \cdot) $ with identity $F_{53}$ (RC \, identity ) $ yz\cdot xx  = y(z x \cdot x) $ is a loop.
\end{corollary}
\begin{proof}
By Theorem \ref{12-parastrophes_1} $F_{53} = (F_{41})^{\ast}$.
\end{proof}

\begin{theorem} \label{Identity_F42}
Quasigroup $(Q, \cdot)$ with identity $F_{42}$    ($ xx\cdot yz = (xx \cdot y) z $) is a left loop.
\end{theorem}
\begin{proof}
We prove that quasigroup $(Q, \cdot)$ with identity $F_{42}$  has left   identity element. In identity $F_{42}$ we put $y:= e_{xx}$. Then we have:
\begin{equation} \label{F42_1}
xx\cdot e_{xx}z = (xx \cdot e_{xx}) z, \; xx\cdot e_{xx}z = xx \cdot z, e_{xx}z = z.
\end{equation}

The following example demonstrates that left  loop with identity $F_{42}$ has
no right  identity element.

\[
\begin{array}{c|ccc}
\cdot & 0 & 1 & 2\\
\hline
   0 & 1 & 2 & 0 \\
    1 & 0 & 1 & 2 \\
    2 & 2 & 0 & 1
\end{array}
\]
\end{proof}

\begin{corollary} \label{identity_F_54}
Quasigroup $(Q, \cdot) $ with identity $F_{54}$  ($ yz\cdot xx  = y(z \cdot x  x) $) is a right loop.
\end{corollary}
\begin{proof}
By Theorem \ref{12-parastrophes_1} $F_{54} = (F_{42})^{\ast}$.
\end{proof}

\begin{theorem} \label{Identity_F43}
Quasigroup $(Q, \cdot)$ with identity $F_{43}$   ( $ xx\cdot yz  = x(x \cdot yz) $) is a left loop.
\end{theorem}
\begin{proof}
We prove that quasigroup $(Q, \cdot)$ with identity $F_{43}$  has left   identity element. In identity $F_{43}$ we put $x = y = f_{z}$. Then we have:
\begin{equation} \label{F43_1}
f_{z}f_{z}\cdot f_{z}z  = f_{z}(f_{z} \cdot f_{z}z), \; f_{z}f_{z}\cdot z  = f_z z, \; f_{z}f_{z} = f_z.
\end{equation}

Further, in identity $F_{43}$ we put $x = f_{x}$. Then we have:
\begin{equation} \label{F43_2}
f_{x}f_{x}\cdot yz  = f_{x}(f_{x} \cdot yz), \,  \textrm{using \, equality \, (\ref{F43_1}),} \, f_{x}\cdot yz  = f_{x}(f_{x} \cdot yz),\;
 yz  = f_{x} \cdot yz.
\end{equation}

The following example demonstrates that left  loop with identity $F_{43}$ has
no right  identity element.

\[
\begin{array}{c|cccccc}
\cdot & 0 & 1 & 2 & 3 & 4 & 5\\
\hline
 0 & 1 & 0 & 3 & 2 & 5 & 4 \\
    1 & 0 & 1 & 2 & 3 & 4 & 5 \\
    2 & 2 & 3 & 4 & 5 & 0 & 1 \\
    3 & 3 & 2 & 5 & 4 & 1 & 0 \\
    4 & 4 & 5 & 0 & 1 & 2 & 3 \\
    5 & 5 & 4 & 1 & 0 & 3 & 2
\end{array}
\]
\end{proof}

\begin{corollary} \label{identity_F_51}
Quasigroup $(Q, \cdot) $ with identity $F_{51}$  ($ yz\cdot xx  = (yz\cdot x)x $) is a right loop.
\end{corollary}
\begin{proof}
By Theorem \ref{12-parastrophes_1} $F_{51} = (F_{43})^{\ast}$.
\end{proof}

\begin{theorem} \label{Identity_F44}
Quasigroup $(Q, \cdot)$ with identity $F_{44}$   ( $ xx\cdot yz  = x(x y \cdot z) $) is a left loop.
\end{theorem}
\begin{proof}
We prove that quasigroup $(Q, \cdot)$ with identity $F_{44}$  has left   identity element. In identity $F_{44}$ we put $x = f_{y}$. Then we have:
\begin{equation} \label{F44_1}
f_{y}f_{y}\cdot yz  = f_{y}(f_{y}y  \cdot z), \; f_{y}f_{y}\cdot yz  = f_{y}\cdot y z, \; f_{y}f_{y}  = f_{y}.
\end{equation}
Further we put  $x=f_x$ in identity $F_{44}$. Then we have:
\begin{equation} \label{F44_2}
f_{x}f_{x}\cdot yz  = f_{x}(f_{x}y  \cdot z), \textrm{we use (\ref{F44_1})},\:
 f_{x}\cdot yz  = f_{x}(f_{x}y  \cdot z), \;   yz  = f_{x}y  \cdot z, \; y  = f_{x}y.
\end{equation}

The following example demonstrates that left  loop with identity $F_{44}$ has
no right  identity element.
\[
\begin{array}{c|ccc}
\cdot & 0 & 1 & 2\\
\hline
    0 & 1 & 2 & 0 \\
    1 & 0 & 1 & 2 \\
    2 & 2 & 0 & 1
\end{array}
\]
\end{proof}

\begin{theorem} \label{Identity_F45}
Quasigroup $(Q, \cdot)$ with identity $F_{45}$   ( $ (x \cdot xy)z = (xx \cdot y)z $ ) is a left loop.
\end{theorem}
\begin{proof}
We prove that quasigroup $(Q, \cdot)$ with identity $F_{45}$  has left   identity element. In identity $F_{45}$ we make cancellation from the right side.  Then we have:
\begin{equation} \label{F45_1}
x \cdot xy = xx \cdot y.
\end{equation}
In identity (\ref{F45_1}) we put $x=f_y$ and obtain
\begin{equation} \label{F45_2}
f_y \cdot f_yy = f_yf_y \cdot y, f_y  = f_yf_y.
\end{equation}
Further we put  $x=f_x$ in identity (\ref{F45_1}). We have:
\begin{equation} \label{F45_3}
f_x \cdot f_xy = f_xf_x \cdot y, \; \textrm{we use (\ref{F45_2})},\: f_x \cdot f_xy = f_x y, f_xy = y.
\end{equation}

The following example demonstrates that left  loop with identity $F_{45}$ has
no right  identity element.

\[
\begin{array}{c|cccccc}
\cdot & 0 & 1 & 2 & 3 & 4 & 5\\
\hline
    0 & 1 & 0 & 3 & 2 & 5 & 4 \\
    1 & 0 & 1 & 2 & 3 & 4 & 5 \\
    2 & 2 & 3 & 4 & 5 & 0 & 1 \\
    3 & 3 & 2 & 5 & 4 & 1 & 0 \\
    4 & 4 & 5 & 0 & 1 & 2 & 3 \\
    5 & 5 & 4 & 1 & 0 & 3 & 2
\end{array}
\]
\end{proof}

\begin{corollary} \label{identity_F_60}
Quasigroup $(Q, \cdot) $ with identity $F_{60}$  ( $  y(zx \cdot x) = y(z \cdot xx) $ ) is a right loop.
\end{corollary}
\begin{proof}
By Theorem \ref{12-parastrophes_1} $F_{60} = (F_{45})^{\ast}$.
\end{proof}

\begin{theorem} \label{Identity_F46}
Quasigroup $(Q, \cdot)$ with identity $F_{46}$ (LC  identity) $ (x\cdot xy)z = x(x \cdot yz) $ does not have left and right unit.
\end{theorem}
\begin{proof}
\[
\begin{array}{c|ccc}
\cdot & 0 & 1 & 2\\
\hline
    0 & 1 & 0 & 2 \\
    1 & 0 & 2 & 1 \\
    2 & 2 & 1 & 0
\end{array}
\]
\end{proof}

\begin{corollary} \label{identity_F_56}
Quasigroup $(Q, \cdot) $ with identity $F_{56}$ (RC identity)  $ (yz \cdot x)x  = y (zx\cdot x) $ does not have left and right unit.
\end{corollary}
\begin{proof}
By Theorem \ref{12-parastrophes_1} $F_{56} = (F_{46})^{\ast}$.
\end{proof}

\begin{theorem} \label{Identity_F47}
Quasigroup $(Q, \cdot)$ with identity $F_{47}$  ($ (x\cdot xy)z = x(xy \cdot z) $)  is a group.
\end{theorem}
\begin{proof}
If in identity $F_{47}$ we denote term $xy$ by variable $t$, then we have $ xt \cdot z = x \cdot tz $.
\end{proof}

\begin{corollary} \label{identity_F_58}
Quasigroup $(Q, \cdot) $ with identity $F_{58}$   ($ (y \cdot zx)x = y (zx\cdot x) $) is a group.
\end{corollary}
\begin{proof}
By Theorem \ref{12-parastrophes_1} $F_{58} = (F_{47})^{\ast}$.
\end{proof}

\begin{theorem} \label{Identity_F48}
Quasigroup $(Q, \cdot)$ with identity $F_{48}$  ($ (xx \cdot y)z = x(x \cdot yz)$) is a left loop.
\end{theorem}
\begin{proof}
We prove that quasigroup $(Q, \cdot)$ with identity $F_{48}$  has left   identity element. In identity $F_{48}$ we put $x = f_{yz}$. Then we have:
\begin{equation} \label{F48_1}
(f_{yz}f_{yz} \cdot y)z = f_{yz}(f_{yz} \cdot yz), \; (f_{yz}f_{yz} \cdot y)z = yz, \; f_{yz}f_{yz} \cdot y = y.
\end{equation}

The following example demonstrates that left  loop with identity $F_{48}$ has
no right  identity element.

\[
\begin{array}{c|cccccc}
\cdot & 0 & 1 & 2 & 3 & 4 \\
\hline
  0 & 1 & 4 & 3 & 0 & 2 \\
    1 & 3 & 0 & 4 & 2 & 1 \\
    2 & 0 & 1 & 2 & 3 & 4 \\
    3 & 2 & 3 & 1 & 4 & 0 \\
    4 & 4 & 2 & 0 & 1 & 3
\end{array}
\]
\end{proof}

\begin{corollary} \label{identity_F_57}
Quasigroup $(Q, \cdot) $ with identity $F_{57}$   ($ (yz\cdot x)x  = y (z \cdot xx) $) is a right loop.
\end{corollary}
\begin{proof}
By Theorem \ref{12-parastrophes_1} $F_{57} = (F_{48})^{\ast}$.
\end{proof}

\begin{theorem} \label{Identity_F49}
Quasigroup $(Q, \cdot)$ with identity $F_{49}$  ($ (xx \cdot y)z = x(xy \cdot z) $) is a left loop.
\end{theorem}
\begin{proof}
We prove that quasigroup $(Q, \cdot)$ with identity $F_{49}$  has left   identity element. In identity $F_{49}$ we change $y\rightarrow (xx)\backslash y$. Then we have:
\begin{equation} \label{F49_1}
(xx \cdot (xx \backslash y)) z\overset{(\ref{(1)})}{=} yz = x((x((xx)\backslash y)) \cdot z).
\end{equation}
Further we use  right division \lq\lq $\backslash $\rq\rq in equality  (\ref{F49_1}).
\begin{equation} \label{F49_2}
(x((xx)\backslash y))z = x\backslash (yz).
\end{equation}
If we put  $x=y$ in equality (\ref{F49_2}), then
we have:
\begin{equation} \label{F49_3}
(x((xx)\backslash x))z = x\backslash (xz) \overset{(\ref{(3)})}{=} z.
\end{equation}

The following example demonstrates that left  loop with identity $F_{49}$ has
no right  identity element.

\[
\begin{array}{c|ccc}
\cdot & 0 & 1 & 2\\
\hline
  0 & 1 & 2 & 0 \\
    1 & 0 & 1 & 2 \\
    2 & 2 & 0 & 1
\end{array}
\]
\end{proof}

\begin{corollary} \label{identity_F_59}
Quasigroup $(Q, \cdot) $ with identity $F_{59}$  ( $ (y \cdot zx)x = y (z\cdot x x) $) is a right loop.
\end{corollary}
\begin{proof}
By Theorem \ref{12-parastrophes_1} $F_{59} = (F_{49})^{\ast}$.
\end{proof}

\begin{theorem} \label{Identity_F44}
Quasigroup $(Q, \cdot)$ with identity $F_{44}$  ($ xx\cdot yz  = x(x y \cdot z) $) is a left loop.
\end{theorem}
\begin{proof}
We put  $x= f_y$ in identity $F_{44}$. Then we have:
\begin{equation} \label{F44_1}
f_yf_y\cdot yz  = f_y(f_y y \cdot z), \, f_yf_y\cdot yz  = f_y \cdot y z, \, f_yf_y  = f_y.
\end{equation}

We put  $x= f_x$ in identity $F_{44}$. Then we have:
\begin{equation} \label{F44_2}
f_xf_x\cdot yz  \overset{(\ref{F44_1})}{=} f_x\cdot yz  = f_x(f_x y \cdot z), \, yz  = f_x y \cdot z, \,  y  = f_x y.
\end{equation}
From the last equality in (\ref{F44_2}) it follows that quasigroup $(Q, \cdot)$ with identity $F_{44}$ has left unit.

The following example demonstrates that left  loop with identity $F_{44}$ has
no right  identity element.
\[
\begin{array}{c|ccc}
\cdot & 0 & 1 & 2\\
\hline
    0 & 1 & 2 & 0 \\
    1 & 0 & 1 & 2 \\
    2 & 2 & 0 & 1
\end{array}
\]
\end{proof}

\begin{corollary} \label{identity_F_52}
Quasigroup $(Q, \cdot) $ with identity $F_{52}$  ($ yz\cdot xx  = (y\cdot z x)x $) is a right loop.
\end{corollary}
\begin{proof}
By Theorem \ref{12-parastrophes_1} $F_{52} = (F_{44})^{\ast}$.
\end{proof}

\section{Type of classical Bol-Moufang  identity}

We try to find some invariants of Bol-Moufang identities which demonstrate (without any additional researches) that  a quasigroup  with this identity (in what this identity is true) has left (right, middle) unit element, or it is  a loop, a group.


We define type of classical Bol-Moufang identity.

\begin{definition}
The order of execution of operations in the left and right side of a quasigroup identity we name a type of the identity. \end{definition}

In Table 1 we also indicate the places of double variable $x$. For example, in the identity
$F_1$   $xy\cdot zx = (xy\cdot z)x$ these places are $\{1,  \, 4\}$.

\section{Table}

\renewcommand{\arraystretch}{1.3}

\begin{longtable}{|p{1.cm}|p{1.9cm}|p{3.cm}|p{0.4cm}|P{0.4cm}|p{0.4cm}|p{0.41cm}|p{1.54cm}|}
\caption{Units in quasigroups with Bol-Moufang identities}
\label{tab:table3}\\
\hline
\textbf{Name} & \textbf{Abbrev.} & \textbf{Identity} & \textbf{f} & \textbf{e} & \textbf{Lo.} & \textbf{Gr.} & \textbf{Type} \\
\hline
$F_1$ &   &  $xy\cdot zx = (xy\cdot z)x$ & + & + & +& + & $(23)=\varepsilon$, \{1,  4\}\\
\hline
$F_3$ &   & $xy\cdot zx = x(y\cdot zx)$ & +& +& +& +& (23)=(13), \{1,  4\}\\
\hline
$F_5$ &   & $(xy \cdot z)x = (x \cdot y z)x$ & + & + & + & +& $\varepsilon = (12)$, \{1,  4\}\\
\hline
$F_{10}$ &   & $x(y\cdot zx) = x(y z \cdot x)$ & + & + & +&  + & $(13)=(132)$, \{1,  4\}\\
\hline
$F_{11}$ &   & $xy\cdot xz = (xy \cdot x)z$ & + & + & + & +& $(23)=\varepsilon$, \{1,  3\}\\
\hline
$F_{12}$ &   & $xy\cdot xz = (x\cdot y x)z$ & + & + & + & +& (23)=(12), \{1,  3\})\\
\hline
$F_{14}$  &   & $xy\cdot xz = x(y \cdot x z)$ & + & + & + & +& (23)=(13), \{1,  3\}\\
\hline
$F_{18}$  &   & $(x\cdot y  x)z = x(y   x \cdot z)$ & + & + & + & + &(12)=(132), \{1,  3\}\\
\hline
$F_{20}$  &   & $x(y  x \cdot z) = x(y \cdot  x  z)$ & + & + & + &  + & (132)=(13), \{1,  3\} \\
\hline
$F_{21}$   &  & $yx \cdot zx = (yx\cdot z)x$  & + & + & + & +& $(23)=\varepsilon$, \{2,  4\}\\
\hline
$F_{23}$   &   & $yx \cdot zx = y ( x z \cdot x)$ & + & + & + & +& $(23)= (132)$, \{2,  4\}\\
\hline
$F_{24}$  &   & $yx \cdot zx = y ( x \cdot z  x)$ & + & + & + & + & $(23)= (13)$, \{2,  4\} \\
\hline
$F_{25}$    &    & $(yx \cdot z)x = (y \cdot x z)x$ & + & + & + & + & $\varepsilon = (12)$, \{2,  4\}\\
\hline
$F_{28}$    &   & $(y \cdot xz)x = y(xz\cdot x)$ & + & + & + & + & $(12) = (132)$, \{2,  4\} \\
\hline
$F_{31}$    &   & $yx \cdot xz = (yx \cdot x)z$ & + & + & + & + & $(23) = \varepsilon $, \{2,  3\}\\
\hline
$F_{32}$   &   & $yx \cdot xz = (y\cdot x x)z$ & + & + & + & + & $(23) = (12) $, \{2,  3\} \\
\hline
$F_{33}$   &   & $yx \cdot xz = y( x x \cdot z)$ & + & + & + & + & $(23) = (132)$, \{2,  3\} \\
\hline
$F_{34}$   &   & $yx \cdot xz = y(x \cdot  x z)$  & + & + & + & + & $(23) = (13)$,  \{2,  3\} \\
\hline
$F_{47}$   &  & $(x\cdot xy)z = x(xy \cdot z)$ & + & + & + & + & $(12) = (132)$, \{1,  2\}\\
\hline
$F_{50}$   &  & $x(x\cdot y z) = x(xy\cdot z)$ & + & + & + &  + & $(13) = (132)$,  \{1, 2\}\\
\hline
$F_{55}$   & & $(yz \cdot x) x = (y \cdot zx)x$  & + & + & + & + & $\varepsilon  = (12)$, \{3,  4\} \\
\hline
$F_{58}$   &  & $(y \cdot zx)x = y (zx\cdot x)$ & + & + & + & + & $(12) = (132)$,  \{3,  4\}\\
\hline
$F_4$ & middle\,\, Moufang  & $xy\cdot zx = x(yz\cdot x)$ & + & + & + & - & (23)=(132), \{3,  4\} \\
\hline
 $F_2$ & middle\,\, Moufang  &  $xy\cdot zx = (x\cdot yz)x$ & + & + &+ & -& (23)=(12),  \{3,  4\}\\
\hline
$F_6$ &  extra \,\, identity & $(xy \cdot z)x = x (y  \cdot z x)$ & + & + & + & -& $\varepsilon = (13)$, \{1,  4\}\\
\hline
$F_{13}$ &  extra\,\, identity  &  $xy\cdot xz = x(y x\cdot z)$ & + & + & + & - & (23)= (132),  \{1,  3\}\\
\hline
$F_{17}$  & left \,\, Moufang   & $(xy \cdot x)z = x(y \cdot  x z)$ & + & + & + & - & $\varepsilon = (13)$, \{1,  3\} \\
\hline
$F_{22}$   & extra \,\, identity   & $yx \cdot zx = (y \cdot x z)x$ &  +& +   & + & - &$(23)= (12)$, \{2,  4\} \\
\hline
$F_{27}$    &  right \,\, Moufang  & $(yx \cdot z)x = y (x\cdot  z  x)$ & + & + & + & - & $\varepsilon = (13)$, \{2,  4\} \\
\hline
$F_{38}$  &  & $(y \cdot xx) z = y(xx \cdot z)$ & + & + & + & - & $(12) = (132)$, \{2,  3\} \\
\hline
$F_{41}$   & LC \,\, identity  & $xx \cdot yz = (x \cdot xy)z$ & + & + & + & - & $(23) = (12)$, \{1,  2\} \\
\hline
$F_{53}$   & RC\,\, identity  & $yz\cdot xx  = y(z x \cdot x)$ & + & + & + & - & $(23) = (132)$, \{3,  4\} \\
\hline
$F_7$ &   & $(xy \cdot z)x = x (y z \cdot x)$ & + & - & - & -& $\varepsilon = (132)$, \{1,  4\}\\
\hline
$F_{16}$  &   & $(xy \cdot x)z = x(y  x\cdot z)$ & + & - & - &  -& $\varepsilon = (132)$, \{1,  3\} \\
\hline
$F_{26}$   &  right \,\, Bol  & $(yx \cdot z)x = y (x z \cdot x)$ & +& - & - & - & $\varepsilon = (132)$, \{2,  4\}\\
\hline
$F_{36}$   & RC \,\, identity & $(yx \cdot x) z = y(xx \cdot z)$ & + & - & - & - & $\varepsilon  = (132)$, \{2,  3\}\\
\hline
$F_{40}$   &  & $y (xx\cdot z)  = y (x\cdot xz)$ & + & - & - &  -& $(132) = (13)$,  \{1,  2\} \\
\hline
$F_{42}$   &  & $xx\cdot yz = (xx \cdot y) z $ & + & - & - & - & $(23) = \varepsilon $,  \{1,  2\}  \\
\hline
$F_{43}$   &  & $xx\cdot yz  = x(x \cdot yz)$ & + & - & - & -& $(23) = (13)$,  \{1,  2\}  \\
\hline
$F_{44}$   &  & $xx\cdot yz  = x(x y \cdot z)$ & + & - & - & - & $(23) = (132)$,  \{1,  2\} \\
\hline
$F_{45}$   &  & $(x \cdot xy)z = (xx \cdot y)z$ & + & - & - & - & $(12) = \varepsilon $,  \{1,  2\} \\
\hline
$F_{48}$   & LC\,\, identity  & $(xx \cdot y)z = x(x \cdot yz)$ & +& - & - & - & $\varepsilon  = (13)$, \{1,  2\} \\
\hline
$F_{49}$   &  & $(xx \cdot y)z = x(xy \cdot z)$ & + & - & - &  - & $\varepsilon  = (132)$,  \{1,  2\} \\
\hline
$F_8$ &   & $(x \cdot y z)x = x (y  \cdot z x)$ & - & +  & - & -& (12)=(13),  \{1,  4\} \\
\hline
$F_{19}$  &  left \,\ Bol  & $(x\cdot y  x)z = x(y \cdot  x  z)$ & - & + & - & -& (12)=(13), \{1,  3\} \\
\hline
$F_{29}$    &   & $(y\cdot xz)x  = y(x \cdot zx)$ & - & + & - & - & $(12) = (13)$,  \{2,  4\} \\
\hline
$F_{35}$   &   & $(yx \cdot x) z = (y\cdot xx)z$  & - & + & - & - & $\varepsilon  = (12)$,  \{2,  3\} \\
\hline
$F_{39}$   & LC\,\, identity  & $(y \cdot xx) z = y(x\cdot xz)$ & - & + & - & - & $(12) = (13)$,  \{2,  3\} \\
\hline
$F_{51}$   &  & $yz\cdot xx  = (yz\cdot x)x$ & - & + & - & - & $(23) = \varepsilon $,  \{3,  4\} \\
\hline
$F_{52}$   &  & $yz\cdot xx  = (y\cdot z x)x$ & - & + & - & - & $(23) = (12)$,  \{3,  4\} \\
\hline
$F_{54}$   &  & $yz\cdot xx  = y(z \cdot x  x)$ & - & + & - & - & $(23) = (13)$, \{3,  4\}  \\
\hline
$F_{57}$   & RC\,\, identity & $(yz\cdot x)x  = y (z \cdot xx)$ & - & + & - & - & $\varepsilon  = (13)$,  \{3,  4\} \\
\hline
$F_{59}$   &  & $(y \cdot zx)x = y (z\cdot x x)$ & - & + & - & - & $(12) = (13)$,  \{3,  4\}  \\
\hline
$F_{60}$   &  & $y(zx \cdot x) = y(z \cdot xx)$ & - & + & - & - & $(132) = (13)$,  \{3,  4\} \\
\hline
$F_9$ &   & $(x \cdot y z)x = x (y   z \cdot x)$ & - & -& - & -& (12) = (132),  \{1,  4\} \\
\hline
$F_{15}$  &   & $(xy \cdot x)z = (x\cdot y  x)z$ & - & - & - & -& $\varepsilon = (12)$,   \{1,  3\}\\
\hline
$F_{30}$   &   & $y(xz \cdot x) = y(x \cdot zx)$ & - & - & - &  - & $(132) = (13)$,  \{2,  4\} \\
\hline
$F_{37}$   & C\,\, identity  & $(yx \cdot x)z = y(x\cdot xz)$ &  - & - & - & - & $\varepsilon  = (13)$, \{2,  3\}  \\
\hline
$F_{46}$   & LC \,\, identity & $(x\cdot xy)z = x(x \cdot yz)$ & - & - & - & - & $(12) = (13)$,  \{1,  2\} \\
\hline
$F_{56}$   & RC\,\, identity  & $(yz \cdot x)x  = y (zx\cdot x)$ & - & - & - & - & $\varepsilon  = (132)$, \{3,  4\}  \\
\hline
\end{longtable}


\begin{thebibliography}{10}

\bibitem{AKHTAR}
Reza Akhtar, Ashley Arp, Michael Kaminski, Jasmine~Van Exel, Davian Vernon, and
  Cory Washington.
\newblock \protect{The varieties of Bol-Moufang quasigroups defined by a single
  operation}.
\newblock {\em Quasigroups Related Systems}, 20(1):1--10, 2012.

\bibitem{VD}
V.D. Belousov.
\newblock {\em Foundations of the Theory of Quasigroups and Loops}.
\newblock Nauka, Moscow, 1967.
\newblock (in Russian).

\bibitem{BIRKHOFF_ENG}
G.~Birkhoff.
\newblock {\em Lattice Theory, Third edition}, volume XXV American Mathematical
  Society.
\newblock American Mathematical Society Colloquium Publications, Providence,
  R.I., 1967.

\bibitem{BIRKHOFF}
G.~Birkhoff.
\newblock {\em Lattice Theory}.
\newblock Nauka, Moscow, 1984.
\newblock (in Russian).

\bibitem{BRUCK_60_1}
R.H. Bruck.
\newblock \protect{Some theorems on Moufang loops}.
\newblock {\em Math. Z.}, 73:59--78, 1960.

\bibitem{Chein_Hala}
Orin Chein and Hala Orlik-Pflugfelder.
\newblock \protect{The smallest Moufang loop}.
\newblock {\em Arch. Math. (Basel)}, 22:573--576, 1971.

\bibitem{Cote}
B.~Cote, B.~Harvill, M.~Huhn, and A.~Kirchman.
\newblock \protect{Classification of loops of generalized Bol-Moufang type}.
\newblock {\em Quasigroups Related Systems}, 19(2):193--206, 2011.

\bibitem{EVANS_49}
T.~Evans.
\newblock Homomorphisms of non-associative systems.
\newblock {\em J. London Math. Soc.}, 24:254--260, 1949.

\bibitem{Fenyves_1}
F.~Fenyves.
\newblock \protect{Extra loops. II. On loops with identities of Bol-Moufang
  type}.
\newblock {\em Publ. Math. Debrecen}, 16:187--192, 1969.

\bibitem{Gagola_10}
Stephen~M. Gagola.
\newblock \protect{Hall's theorem for Moufang loops}.
\newblock {\em J. Algebra}, 323(12):3252--3262, 2010.

\bibitem{Gagola_11}
Stephen~M. Gagola.
\newblock \protect{How and why Moufang loops behave like groups}.
\newblock {\em Quasigroups Related Systems}, 19(1):1--22, 2011.

\bibitem{GALKIN}
V.M. Galkin.
\newblock {\em Quasigroups}, volume~26 of {\em Algebra, Topology, Geometry},
  pages 3--44.
\newblock VINITI, Moscow, 1988.
\newblock (in Russian).

\bibitem{HSTY_IMI}
Grigorii Horosh, Victor Shcherbacov, Alexandru Tcachenco, and Tatiana Yatsko.
\newblock \protect{On Groupoids With Classical Bol-Moufang Type Identities}.
\newblock In {\em Proceedings of the Fifth Conference of Mathematical Society
  of Moldova, IMCS-55, September 28 - October 1, 2019}, pages 77--84, Chisinau,
  Republic of Moldova, September 2019.

\bibitem{IZ_SCERB}
V.I. Izbash and V.A. Shcherbacov.
\newblock On quasigroups with \protect{Moufang} identity.
\newblock In {\em Abstracts of The Third International Conference in memory of
  M.I. Kargapolov (1928-1976)}, pages 134--135, Krasnoyarsk, Russian
  Federation, August 1993.
\newblock (in Russian).

\bibitem{JAIEOLA_2}
T.G. Jaíyeola, E.~Ilojide, M.~O. Olatinwo, and F.~Smarandache.
\newblock \protect{On the Classification of Bol-Moufang Type of Some Varieties
  of Quasi Neutrosophic Triplet Loop (Fenyves BCI-Algebras)}.
\newblock {\em Symmetry}, 10:1--16, 2018.
\newblock doi:10.3390/sym10100427.

\bibitem{Kin_KUN_04}
M.K. Kinyon and K.~Kunen.
\newblock The structure of extra loops.
\newblock {\em Quasigroups Related Systems}, 12:39--60, 2004.

\bibitem{KUNEN}
K.~Kunen.
\newblock Moufang quasigroups.
\newblock {\em J. Algebra}, 183:231--234, 1996.

\bibitem{KUNEN_96}
K.~Kunen.
\newblock Quasigroups, loops and associative laws.
\newblock {\em J. Algebra}, 185(1):194--204, 1996.

\bibitem{Kurosh}
A.G. Kurosh.
\newblock {\em Group Theory}.
\newblock Nauka, Moscow, 1967.
\newblock (in Russian).

\bibitem{MAC_CUNE_MACE}
W.~McCune.
\newblock {\em \protect{Mace 4}}.
\newblock University of New Mexico, www.cs.unm.edu/mccune/prover9/, 2007.

\bibitem{MAC_CUNE_PROV}
W.~McCune.
\newblock {\em \protect{Prover 9}}.
\newblock University of New Mexico, www.cs.unm.edu/mccune/prover9/, 2007.

\bibitem{MUFANG}
R.~Moufang.
\newblock \protect{Zur Structur von Alternativ K\" orpern}.
\newblock {\em Math. Ann.}, 110:416--430, 1935.

\bibitem{HOP}
H.O. Pflugfelder.
\newblock {\em Quasigroups and Loops: Introduction}.
\newblock Heldermann Verlag, Berlin, 1990.

\bibitem{Ph_2005}
J.~D. Phillips and Petr Vojtechovsky.
\newblock \protect{The varieties of loops of Bol-Moufang type}.
\newblock {\em Algebra Universalis}, 54(3):259--271, 2005.

\bibitem{Scer_Shcherb_2}
A.V. Scerbacova and V.A. Shcherbacov.
\newblock \protect{On spectrum of medial $T_2$-quasigroups}.
\newblock {\em Bul. Acad. \c Stiin\c te Repub. Mold. Mat.}, (2):143--154, 2016.

\bibitem{VS_VI}
V.A. Shcherbacov and V.I. Izbash.
\newblock On quasigroups with \protect{Moufang} identity.
\newblock {\em Bul. Acad. Stiinte Repub. Mold., Mat.}, (2):109--116, 1998.

\bibitem{2017_Scerb}
Victor Shcherbacov.
\newblock {\em Elements of Quasigroup Theory and Applications}.
\newblock CRC Press, Boca Raton, 2017.

\end{thebibliography}

\vspace{2mm}
\begin{center}
\begin{parbox}{118mm}{\footnotesize
Natalia Didurik$^{1}$,
Victor Shcherbacov$^{2}$,
\vspace{3mm}

\noindent $^{1}$Ph. D. Student/Dimitrie Cantemir University

\noindent Email: natnikkr83@mail.ru

\medskip

\noindent $^{2}$Principal Researcher/Institute of Mathematics and Computer Science of Moldova

\noindent Email: victor.scerbacov@math.md
}
\end{parbox}
\end{center}

\end{document}